\documentclass[12pt]{article}
\usepackage{cmap}
\usepackage[warn]{mathtext}
\usepackage[T2A]{fontenc}
\usepackage[utf8]{inputenc}
\usepackage[main=english, russian]{babel}
\DeclareSymbolFont{T2Aletters}{T2A}{cmr}{m}{it}
\usepackage{amssymb}
\usepackage{tabto}
\usepackage{amsmath}
\usepackage{amsthm}
\usepackage{amsfonts}
\usepackage{xcolor}
\usepackage[colorlinks = true, urlcolor= blue, citecolor = black, linkcolor = black]{hyperref}
\usepackage[a4paper, left=30mm, right=30mm, top=30mm, bottom=20mm]{geometry}
\usepackage{csquotes}
\usepackage{graphicx}
\usepackage{wrapfig}
\usepackage{caption}

\title{Erd\H os Matching Conjecture for almost perfect matchings}
\author{Dmitriy Kolupaev\footnote{Moscow Institute of Physics and Technology, Russia}, Andrey Kupavskii\footnote{Saint-Petersburg State University, Russia and G-SCOP, CNRS, Universit\'e Grenoble-Alpes, France and
Moscow Institute of Physics and Technology, Russia; Email: {\tt kupavskii@ya.ru}. Research of the author is supported by the grant RSF N 22-11-00131.}}

\newtheorem{theorem}{Theorem}[section]
\newtheorem{claim}{Proposition}[section]
\newtheorem{conjecture}{Conjecture}[section]

\theoremstyle{definition}

\numberwithin{equation}{section}

\newtheorem*{thm*}{Theorem}
\newcommand{\ff}{{\mathcal F}}
\newcommand{\aaa}{{\mathcal A}}
\newcommand{\bb}{{\mathcal B}}

\begin{document}
\maketitle

\begin{abstract}
In 1965 Erd\H os asked, what is the largest size of a family of $k$-element subsets of an $n$-element set that does not contain a matching of size $s+1$? In this note, we improve upon a recent result of Frankl and resolve this problem for \(s>101k^{3}\) and \((s+1)k\leqslant n<(s+1)(k+\frac{1}{100k})\).
\end{abstract}
\section{Introduction}
Let \(n, k\geqslant 2\) be positive integers. We denote \([n] = \{1,\ldots, n\}\) and for a set $X$ put $2^{X}$, ${X\choose k}$ to be the power set of $X$ and the set of all $k$-element subsets of $X$, respectively. A {\it family} is simply a collection of sets. For a family $\ff$, let  \(\nu(\mathcal{F})\) stand for the {\it matching number of $\ff$}, that is, the maximum number of pairwise disjoint sets from  \(\mathcal{F}\).

What is the largest family $\ff\subset {[n]\choose k}$ such that $\nu(\ff)\le s$? We assume that $n\ge (s+1)k$.  Otherwise, the question is trivial: the family ${[n]\choose k}$ itself has matching number at most $s$. Here are two natural candidate families:
\begin{equation}
    \mathcal{A}=\textstyle\binom{[(s+1)k-1]}{k},
\end{equation}
\begin{equation}
    \mathcal{B}=\left\{B\in \textstyle\binom{[n]}{k}: B\cap[s] \neq \varnothing\right\}.
\end{equation}
It is not difficult to check that indeed \(\nu(\mathcal{A})=\nu(\mathcal{B})=s\).

In 1965 Erd\H os \cite{E} studied this question and proposed the following conjecture.
\begin{conjecture}[Erd\H os~\cite{E}]
Let $n,k,s$ be positive integers and $n\ge (s+1)k$. Assume that \(\mathcal{F}\subset\binom{[n]}{k}\) satisfies \(\nu(\mathcal{F})\leqslant s\). Then
\begin{equation}\label{ineq1}
    |\mathcal{F}|\leqslant \max\left\{|\mathcal{A}|,|\mathcal{B}|\right\}.
\end{equation}
\end{conjecture}
This conjecture, as well as several of its variants, has been extensively studied. We refer to the works \cite{aletal, FK21, K49} for a survey of the recent developments around the EMC and its rainbow version. Importantly, a majority of the works, e.g., \cite{aletal,BDE, F4,F6, FK21, HLS, KLLM, KLchv, K49}, deal with the regimes when $\bb$ is larger than $\aaa$ (and show extremality of $\aaa$ under more or less restrictive conditions). In particular, in \cite{FK21} Frankl and the second author showed that \(|\mathcal{F}|\leqslant|\mathcal{B}|\) for all sufficiently large $s$ and \(n\geqslant\frac{5}{3}sk-\frac{2}{3}s\).

In this note, however, we are interested in the case when $\aaa$ is extremal or, put differently, the forbidden matching is almost perfect. Standard, albeit tedious, calculations show that for \(n\geqslant (s+1)(k+\frac{1}{2})\) we have
\(|\mathcal{A}| < |\mathcal{B}|\).

In 2017 Frankl \cite{F7} proved the following theorem.
\begin{theorem}[Frankl \cite{F7}]\label{fr1}
Let \(s>k\geqslant 2\) be positive integers. Assume that $n$ satisfies  \((s+1)k\leqslant n<(s+1)(k+k^{-2k-1}/2)\). Then for any  \(\mathcal{F}\subset\binom{[n]}{k}\) with \(\nu(\mathcal{F})\leqslant s\) we have \(|\mathcal{F}|\leqslant|\mathcal{A}|.\)
\end{theorem}

Building upon the approach of Frankl, we can significantly improve the bounds in Theorem~\ref{fr1}. This is the main goal of this note.
\begin{theorem}\label{th2}
Let \(s>k\geqslant 5\) be positive integers, \(s>101k^3\). Assume that $n$ satisfies  \((s+1)k\leqslant n<(s+1)(k+\frac{1}{100k})\). Then for any  \(\mathcal{F}\subset\binom{[n]}{k}\) with \(\nu(\mathcal{F})\leqslant s\) we have \(|\mathcal{F}|\leqslant|\mathcal{A}|.\)
\end{theorem}
One can note that, while the bound on $n$ is significantly weakened, we have to impose a bound on $s$. This was not necessary in Theorem~\ref{fr1} since for  \(s<2k^{2k+1}\) the theorem is trivial: the inequalities on $n$ amount to \(n=(s+1)k\), and in this case the inequality (\ref{ineq1}) is easy and was proved by Kleitman in \cite{Kl}.

We note that a very interesting question in this direction is to replace $1/100k$ by some small constant.

In the proof of Theorem~\ref{th2} we shall follow the proof of  \cite{F7}, and, in particular, shall use much of the same notation and definitions. 
\section{Definitions and set-up}
This section introduces the framework of Frankl from \cite{F7}. The statements given without a proof, if non-obvious, are proved in \cite{F7}.

In what follows, we suppose that $\ff$ is the largest family satisfying the restrictions in Theorem~\ref{th2}.
\subsection{Shifting and Trace}
Take two $k$-sets $F=\{a_1,a_2,\ldots,a_k\}$, $a_1<a_2<\ldots<a_k$, and $G=\{b_1,b_2,\ldots,b_k\}$, $b_1<b_2<\ldots<b_k$. We say that $F$ {\it precedes} $G$, or \(F\preceq G\), if \(a_i\leqslant b_i\) for each \(i\in[k]\). If for any \(F, G\) such that \(F\prec G\) and \(G\in \mathcal{F}\) we have that \(F\in \mathcal{F}\), then we say that the family \(\mathcal{F}\) is \textit{shifted}.

It is standard (cf., e.g., \cite{Fra3}) that we can w.l.o.g. assume that $\ff$ is shifted. We assume this in what follows.

For a family \(\mathcal{F}\) define its \textit{trace} on \([(s+1)k-1]\) as follows:\begin{equation}
    \mathcal{T}=\mathcal{T(F)}=\{T\cap[(s+1)k-1]: T\in\mathcal{F}\}.
\end{equation}
It is not difficult to derive from shiftedness of $\ff$ that \(\nu(\mathcal{T})=s\). By maximality of $\ff$, we may suppose that \(\mathcal{F}=\{F\in\binom{[n]}{k}: \exists T\in\mathcal{T}, T\subset F\}\). The latter allows us to write
\begin{equation}\label{eqF}
    |\mathcal{F}|=\sum\limits_{d=1}^k\left|\{T\in \mathcal{T}: |T|=d\}\right|\cdot\textstyle\binom{\bar n}{k-d},\
\end{equation} where \(\bar n = n-(s+1)k+1\).

In \cite{F7} Frankl proves the following claim.
\begin{claim}\label{spcl}
There exists a set \(G_0\subset[(s+1)k-1]\) of size \(k-1\), \(G_0\notin\mathcal{T}\), such that for any \(B\in\mathcal{F}\), \(B\cap G_0=\varnothing\), we have
\begin{equation}
    G_0\cup\{b\}\in\mathcal{F},
\end{equation}
where \(b\) is the minimal element of \(B\).
\end{claim}

Shiftedness and maximality of $\ff$ imply that there exist pairwise disjoint sets  \(G_1,\ldots,G_s\in\mathcal{T}\) of size \(k\) such that \(G_i\cap G_0 =\varnothing\) for \(i=1,\ldots,s\). In what follows, we fix \(G_0,\ldots,G_s\).
\subsection{Width and weight}
For a non-empty set \(T\subset [(s+1)k-1]\) we define its {\it width} as follows:
\begin{equation}
    v(T)=|\{i\in[s]: T\cap G_i \neq \varnothing\}|.
\end{equation}
It is easy to see that \(v(T)\neq 0\) for \(T\in\mathcal{T}\).

For a non-empty set $T\subset [(s+1)k-1]$ of width $c$ and size $d$ we define its {\it weight} as follows:
\begin{equation}
    w(T)=w_{c,d}=\frac{\binom{\bar n}{k-d}}{\binom{s-c}{k-c}}.
\end{equation}

For \(M=\{m_1,\ldots,m_k\}\subset[s]\) consider
\begin{equation}
    G(M)=G_0\cup G_{m_1}\cup\ldots\cup G_{m_k},
\end{equation}

Denote by \(\mathcal{T}_M\) the family of all sets \(T\) such that \(T\in\mathcal{T}, T\subset G(M)\). An easy double-counting argument shows that  (\ref{eqF}) can be rewritten as follows.
\begin{equation}\label{eqFF}
    |\mathcal{F}|=\sum\limits_{M\in\binom{[s]}{k}} \sum\limits_{T\in\mathcal{T}_M}w(T).
\end{equation}
Denote
\begin{equation}
    w(M)=w_{\mathcal{F}}(M)=\sum\limits_{T\in\mathcal{T}_M(\mathcal{F})}w(T).
\end{equation}

The family \(\mathcal{A}\) also contains the sets  \(G_1,G_2,\ldots,G_s\). Thus, we can define the width and weight for the sets from \(\mathcal{A}\) analogously. We get an equation, analogous to  (\ref{eqFF}):
\begin{equation}\label{eqA}
    |\mathcal{A}|=\sum\limits_{M\in\binom{[s]}{k}}w_{\mathcal{A}}(M)=\sum\limits_{M\in\binom{[s]}{k}}\sum\limits_{\substack{|T|=k\\ T\subset G(M)}}w(T).
\end{equation}

Looking at equalities (\ref{eqFF}) and (\ref{eqA}), it should be clear that in order to prove \(|\mathcal{F}|\leqslant|\mathcal{A}|\), it is sufficient to establish the following inequality for each  \(M\in\binom{[s]}{k}\):
\begin{equation}\label{ineq4}
    w_{\mathcal{F}}(M)\leqslant w_{\mathcal{A}}(M).
\end{equation}
This is what we are going to do in the next section. In what follows, we fix the choice of  \(M\in\binom{[s]}{k}\) and work only with the sets that belong to  \(G(M)\). Assume that \(M=\{m_1,\ldots,m_k\}\). To simplify notation in what follows, we put \(B_i=G_{m_i}\) for \(i=1,\ldots,k\).

For \(1\leqslant c\leqslant d\leqslant k\) put
\begin{equation}
    N_{c,d}=\left|\{T\in\mathcal{T}_M:v(T)=c,|T|=d\}\right|.
\end{equation}

\section{Proof of Theorem~\ref{th2}}
Recall the set-up from the end of the previous section.
\subsection{Bounding the sum of weights}
Put \(\varepsilon=\frac{1}{100k}\) and recall that $\bar n<\varepsilon (s+1)+1$. In this subsection, we give most of the necessary calculations for the remainder of the proof. We start by bounding the weight of a set \(T\) of width \(c\) and cardinality \(d\).
\begin{claim}\label{clWeight}
\begin{equation}
    w_{c,d}\leqslant\left(1+\frac{1}{k}\right)\dfrac{\varepsilon^{k-d}}{s^{d-c}}\frac{(k-c)!}{(k-d)!}.
\end{equation}
\end{claim}
\begin{proof}
\begin{multline*}
    w_{c,d}=\frac{\binom{\bar n}{k-d}}{\binom{s-c}{k-c}}\leqslant\frac{(\varepsilon s)^{k-d}(k-c)!}{s^{k-c}(k-d)!}\cdot\frac{\left(1+\frac{\varepsilon+1}{\varepsilon s}\right)^{k-d}}{\left(1-\frac{k-1}{s}\right)^{k-c}}\leqslant\left(1+\frac{1}{k}\right)\dfrac{\varepsilon^{k-d}}{s^{d-c}}\frac{(k-c)!}{(k-d)!}.
\end{multline*}
The last inequality is implied by the following chain of inequalities:
\begin{multline*}
    \frac{\left(1-\frac{k-1}{s}\right)^{k-1}}{\left(1+\frac{\varepsilon+1}{\varepsilon s}\right)^{k-1}}\geqslant\left(1-\frac{(k-1)^2}{s}\right)\left(1-\frac{(k-1)(\varepsilon+1)}{\varepsilon s}\right)\geqslant\\
    \geqslant1-\frac{(k-1)^2}{s}-\frac{(k-1)(100k+1)}{s}\geqslant1-\frac{1}{k+1}=\frac{1}{1+\frac{1}{k}}.
    \qedhere
\end{multline*}
\end{proof}

\begin{claim}
\begin{equation}
    N_{c,d}\leqslant\binom{k}{c}\frac{k^{2d-c}}{(d-c)!}.
\end{equation}
\end{claim}
\begin{proof}
For each set \(T\in\mathcal{T}_M\) that intersects exactly \(c\) sets out of \(B_1,\ldots,B_k\) there is a subset $\{W_1,\ldots, W_c\}\subset \{B_1,\ldots, B_k\}$ and non-negative integers $a_0,a_1,\ldots, a_c$ such that $T$ intersects $W_i$ in $a_i+1$ elements and $G_0$ in $a_0$ elements, moreover, we have  \(a_0+\ldots+a_c=d-c\). We can choose $W_1,\ldots, W_c$ in \(\binom{k}{c}\) different ways. Thus
\begin{align*}
    &N_{c,d}\leqslant\binom{k}{c}\sum\limits_{\substack{a_0,\ldots,a_c:\ a_i\geqslant 0, \\a_0+\ldots+a_c=d-c}}\binom{k-1}{a_0}\binom{k}{a_1+1}\ldots\binom{k}{a_c+1}\leqslant\\
    &\leqslant\binom{k}{c}\sum\limits_{a_0+\ldots+a_c=d-c}k^c\binom{k-1}{a_0}\binom{k-1}{a_1}\ldots\binom{k-1}{a_c}=\\
    &=k^c\binom{k}{c}\binom{(k-1)(c+1)}{d-c}\leqslant k^c\binom{k}{c}{k^2\choose d-c}\leqslant \binom{k}{c}\frac{k^{2d-c}}{(d-c)!}.
    \qedhere
\end{align*}
\end{proof}

For a non-negative integer $g$, denote by \(W_g\) the sum of weights of the sets  \(T\in\mathcal{T}_M\) such that \(|T|<k\) and  \(|T|-v(T)\geqslant g\). Note that  $W_0$ is the total weight of sets in $G(M)$ that lie in $\mathcal F\setminus \mathcal A$. As $g$ grows, their structure is harder to control, but, luckily, their contribution decays exponentially (in our regime). 
\begin{claim}\label{weightE}
For each \(g=0,\ldots,k-2\) we have
\begin{equation}\label{ineq33}
    W_g\leqslant\left(1+\frac{2}{k}\right)\frac{k^{k+2g}}{s^g g!}\varepsilon.
\end{equation}
\begin{proof}
Denote \(U_g=\left(1+\frac{1}{k}\right)\frac{k^{k+2g}}{s^g g!}\varepsilon\). Then for \(d=c+g,\ldots,k-1\) we have
\begin{multline*}
    w_{c,d}N_{c,d}\leqslant\left(1+\frac{1}{k}\right) \frac{\varepsilon^{k-d}}{s^{d-c}}\frac{(k-c)!}{(k-d)!}\binom{k}{k-c}\frac{k^{2d-c}}{(d-c)!}\leqslant\\
    \leqslant\left(1+\frac{1}{k}\right)\frac{\varepsilon^{k-d}}{s^{d-c}}k^{k-c}\frac{k^{2d-c}}{g!}=U_g\cdot\varepsilon^{k-d-1}\left(\frac{k^2}{s}\right)^{d-c-g}.
\end{multline*}
Here we used that \(\binom{k}{c}=\binom{k}{k-c}\leqslant\frac{k^{k-c}}{(k-c)!}\).

Next, we can bound \(W_g\) as follows:
\begin{multline*}
    W_g=\sum\limits_{c=1}^{k-g-1}\sum\limits_{d=c+g}^{k-1}w_{c,d}N_{c,d}\leqslant U_g\sum\limits_{c=1}^{k-g-1}\sum\limits_{d=c+g}^{k-1}\varepsilon^{k-d-1}\left(\frac{k^2}{s}\right)^{d-c-g}\leqslant\\
    \leqslant U_g\left(\sum\limits_{i=0}^{\infty}\varepsilon^i\right)\left(\sum\limits_{j=0}^{\infty} \left(\frac{k^2}{s}\right)^j\right)=U_g\frac{1}{1-\varepsilon}\frac{1}{1-\frac{k^2}{s}}.
\end{multline*}
It remains to note that  \(\left(1+\frac{1}{k}\right)\left(1-\frac{1}{100k}\right)^{-1} \left(1-\frac{1}{101k}\right)^{-1} <1+\frac{2}{k} \), and so the inequality (\ref{ineq33}) is proved.
\end{proof}
\end{claim}

We now study the behavior of sets from \(T\in\mathcal{T}_M\) that satisfy \(v(T)+1=|T|<k\) more precisely, to get finer control. Denote by \(\mathcal{R}_d\), $\mathcal X_d$ the family of sets in \(\mathcal{T}_M\) of size \(d\) and width \(d-1\) that do and do not intersect \(G_0\), respectively. Put \(r_d=|\mathcal{R}_d|\) and \(x_d=|\mathcal{X}_d|\). It should be clear that \(r_d+x_d=N_{d-1,d}\). We denote by \(R\) the sum of weights of sets in $\mathcal R_i$, $i\in[k-1]$. Analogously, we denote by $X$ the sum of weights of sets in $\mathcal X_i$, $i\in [k-1]$.
\begin{claim}
\begin{equation}
    R\leqslant\left(1+\frac{2}{k}\right)\frac{2\varepsilon}{s}r_{k-1},
\end{equation}
\begin{equation}
    X\leqslant\left(1+\frac{2}{k}\right)\frac{2\varepsilon}{s}x_{k-1}.
\end{equation}
\end{claim}
\begin{proof}
Since \(\mathcal{F}\) is shifted, for any \(d=2,\ldots,k-2\) each set from \(\mathcal{R}_d\) lies in exactly  \(k(k-d+1)\) sets from \(\mathcal{R}_{d+1}\). At the same time, each set in  \(\mathcal{R}_{d+1}\) contains at most \(d\) sets from \(\mathcal{R}_{d}\). A simple double counting argument shows that \(k(k-d+1)r_d\leqslant dr_{d+1}\), which in particular implies \begin{equation}
    r_d\leqslant r_{k-1}
\end{equation}
for every \(d=2,\ldots,k-1\).
Then we can bound
\begin{align*}
    R=&\sum\limits_{d=2}^{k-1}w_{d-1,d}r_d\leqslant\sum\limits_{d=2}^{k-1}w_{d-1,d}r_{k-1} \\
    \le& \left(1+\frac{1}{k}\right)\sum\limits_{d=2}^{k-1}\frac{1}{s}\varepsilon^{k-d}(k-d+1)r_{k-1}\le \left(1+\frac{2}{k}\right)\frac{2\varepsilon}{s}r_{k-1},\\
\end{align*}
because $\varepsilon = 1/100k$ and $(1+\frac 1k)\sum_{d=2}^{k-1}(1/100k)^{k-d-1}(k-d+1)\le 2\left(1+\frac{2}{k}\right)$. The second inequality is proved analogously.
\end{proof}

Put \(W=W_0\) to be the sum of weights of sets in \(T\in\mathcal{T}_M\) such that \(|T|<k\). Recall that this is exactly the sum of weights of sets that contribute to $\ff$ and not to $\aaa$.

Also recall that we have fixed \(M\) and only consider sets that lie in \(G(M)\). In what follows, we are going to bound \(W\) under the assumption that 
\begin{equation}\label{ineqM}
    w_{\mathcal{F}}(M)\geqslant w_{\mathcal{A}}(M),
\end{equation}
since otherwise the inequality (\ref{ineq4}) is satisfied.
\subsection{Full transversals}
We call a set \(T\) \textit{a full transversal} if 
\begin{equation}
    v(T)=|T|=k.
\end{equation}
It is not difficult to check that the weight of a full transversal equals $1$, and the number of different full transversals is \(k^k\).
\begin{claim}\label{clTT}
There are no sets in \(\mathcal{T}_M\) such that \(v(T)=|T|<k\). In other words, 
\begin{equation}
    N_{d,d}=0
\end{equation}
for \(d=1,\ldots,k-1\).
\end{claim}
\begin{proof}
From Proposition~\ref{weightE} we get 
\begin{align*}
    W = W_0\leqslant\left(1+\frac{2}{k}\right)k^k\varepsilon.
\end{align*}
We assumed that inequality (\ref{ineqM}) holds, and thus all but at most \(\left(1+\frac{2}{k}\right)k^k\varepsilon<(k-1)^{k-1}\) full transversals must belong to \(\mathcal{T}_M\).

Let us assume the contrary, that is, there is a set \(T\in\mathcal{T}_M\) such that \(v(T)=|T|<k\). Since $\ff$ is shifted, we can w.l.o.g. assume that \(|T|=k-1\). W.l.o.g. assume that \(T\cap B_k=\varnothing\).

Let us introduce convenient notations for the elements of \(B_i\):
\begin{align*}
    &B_i=\left\{b_i^1,\ldots,b_i^k\right\},\\
    &\tilde B_i=B_i\setminus \left\{b_i^k\right\},\\
    &T=\left\{b_1^k,\ldots,b_{k-1}^k\right\}.
\end{align*}
Also, assume that \(b_k^k\) is the minimal element of \(B_k\).

A \textit{cyclic shift} on the set \(B\) with fixed numeration \(b_1,\ldots,b_N\) of its elements is a permutation of the form \((b_1\ \ldots\ b_N)^j\) in cycle notation, where $j$ is one of \(0,\ldots,N-1\).
For an arbitrary collection \(\overline{\sigma}=(\sigma_1,\ldots,\sigma_{k-1})\) of cyclic shifts on the sets \(\tilde B_1,\ldots,\tilde B_{k-1}\) (with arbitrary numeration) consider transversals
\begin{align*}
    T_{i,\overline{\sigma}}=\left\{\sigma_1(b_1^i),\ldots,\sigma_{k-1}(b_{k-1}^i),b_k^i\right\}
\end{align*}
for \(i=1,\ldots,k-1\). Transversals \(T_{1,\overline{\sigma}},\ldots,T_{k-1,\overline{\sigma}}\) are pairwise disjoint, moreover, the collections  \(\left\{T_{i,\overline{\sigma}}\right\}_{i=1}^{k-1}\) themselves do not have common transversals for different \(\overline{\sigma}\). There are \((k-1)^{k-1}\) such collections in total, and thus there is a \(\overline{\sigma}\) such that \(T_i=T_{i,\overline{\sigma}}\in\mathcal{T}_M\) for all \(i=1,\ldots,k-1\).

From Proposition~\ref{spcl} we get that \(G_0\cup\left\{b_k^k\right\}\in\mathcal{T}\). Therefore, the sets  \(T,T_1,\ldots,\) \(T_{k-1},G_0\cup\left\{b_k^k\right\}\) together with \(G_i\), \(i\notin M\), are pairwise disjoint and belong to $\mathcal T$. Since there are \(s+1\) sets in this collection, we get a contradiction with $\nu(\mathcal T) \le s$. 
\end{proof}

Propositions~\ref{clTT} and~\ref{weightE} imply that 
\begin{equation}
    W=W_1\leqslant\left(1+\frac{2}{k}\right)\frac{k^{k+2}\varepsilon}{s}.
\end{equation}
Therefore, \(\mathcal{T}_M\) contains all but at most \(\left(1+\frac{2}{k}\right)\frac{k^{k+2}\varepsilon}{s}<(k-2)^{k-1}\) full transversals. Being more accurate,  we can write \(W=R+S+W_2\), and thus
\begin{equation}\label{ineqAlmF}
    W\leqslant\left(1+\frac{2}{k}\right)\frac{2\varepsilon}{s}\left(r_{k-1}+x_{k-1}+\frac{k^{k+4}}{4s}\right).
\end{equation}
\subsection{Almost full transversals}
We call a set $T$ an \textit{almost full transversal} if 
\begin{equation}
    |T|-v(T)=1,\ |T|=k.
\end{equation}
The weight of an almost full transversal equals \(\frac{1}{s-k+1}>\frac{1}{s}\). Since \(w_{\mathcal{F}}(M)\geqslant w_{\mathcal{A}}(M)\), all but at most $Ws$ almost full transversals are present in  \(\mathcal{T}_M\). Taking into account that \(r_{k-1}\leqslant\frac{(k-1)^2k^{k-1}}{2}\), as well as (\ref{ineqAlmF}) and (\ref{ineqM}), we get that \(W s\leqslant\left(1+\frac{2}{k}\right)2\varepsilon\left(x_{k-1}+k^{k+1}\right)\).
\begin{claim}
\begin{equation}
    x_{k-1}\leqslant\left(1+\frac{3}{k}\right)\varepsilon k^{k+1}.
\end{equation}
\end{claim}
\begin{proof}
We say that the set \(T\in\mathcal{X}_{k-1}\) has type \(\left(l,\{i,j\}\right)\) if \(|T\cap B_l|=2\), \(|T\cap B_i|=|T\cap B_j|=0\). An almost full transversal $Q$ is called a \textit{mask} of type \(\left(l,f\right)\) if \(|Q\cap B_l|=0\), \(|Q\cap B_f|=2\).

We say that a set \(T\in\mathcal{X}_{k-1}\) of type \(\left(l,\{i,j\}\right)\) and mask \(Q\) form a \textit{bad pair } \(\left(T,Q\right)\) if (i) \(Q\) has one of the types \(\left(l,i\right)\) or \(\left(l,j\right)\), (ii) moreover, \(T\cap Q=\varnothing\) and (iii) the set \(B_m\) that intersect \(Q\) in $1$ element  but does not intersect \(T\) (that is, \(m=i\) or \(m=j\)), intersect \(Q\) in an element that is not minimal in $B_m$.

For any set \(T\in\mathcal{X}_{k-1}\) and a bad pair \(\left(T,Q\right)\) the mask \(Q\) does not belong to \(\mathcal{T}_M\). Indeed, w.l.o.g. assume that \(T\) has type \(\left(l,\{i,j\}\right)\) and \(Q\) has type  \(\left(l,i\right)\). Denote by $b$ the minimal element of \(B_j\). Using practically the same reasoning with cyclic shifts as in the proof of Proposition~\ref{clTT} and the fact that \(\mathcal{T}_M\) contains all but at most \(W_1<(k-2)^{k-1}\) full transversals, we can find full transversals  \(T_i\in\mathcal{T}_M\), \(i=1,\ldots,k-2\), such that the sets \(T,T_1,\ldots,T_{k-2},Q,G_0\cup\left\{b\right\}\) are pairwise disjoint. If \(Q\in\mathcal{T}_M\) then, together with \(G_i\), \(i\notin M\), we get \(s+1\) pairwise disjoint sets in $\mathcal T$, a contradiction.

Denote by \(m_k\) the number of masks \(Q\), for which there exists at least $1$ set  \(T\in\mathcal{X}_{k-1}\) such that the pair \(\left(T,Q\right)\) is bad. For any given mask \(Q\) the number of such sets \(T\) is at most \(\frac{k(k-1)^{k-2}(k-2)}{2}\). Indeed, let \(Q\) have type \(\left(l,i\right)\). Then we first fix  \(j\notin \left\{i,l\right\}\) and count the number of corresponding sets \(T\) of type \(\left(l,\{i,j\}\right)\). We can choose two elements in $B_l$ in $\frac {k(k-1)}2$ ways and the elements from the remaining $(k-3)$ sets $B_f$ in $(k-1)^{k-3}$ ways.

Given a set \(T\in\mathcal{X}_{k-1}\), the number of masks that form a bad pair with  \(T\) is  \(k(k-1)^{k-1}\). Therefore, an easy double counting implies
\begin{equation}
    2x_{k-1}\leqslant m_k.
\end{equation}
We have already mentioned that all such masks do not  belong to \(\mathcal{T}_M\). Thus, we get 
\begin{align*}
    2x_{k-1}\leqslant m_k\leqslant\left(1+\frac{2}{k}\right)2\varepsilon\left(x_{k-1}+k^{k+1}\right),
\end{align*}
which, after rearranging and applying some easy bounds, implies the proposition. 
\end{proof}
The last proposition, together with the bound on \(r_{k-1}\) and (\ref{ineqAlmF}), implies that 
\begin{equation}
    W\leqslant\left(1+\frac{3}{k}\right)\frac{2\varepsilon k^{k+1}}{s}.
\end{equation}
Therefore, \(\mathcal{T}_M\) contains all but at most  \(\left(1+\frac{3}{k}\right)2\varepsilon k^{k+1}\) transversals and full transversals.
\begin{claim}\label{clRX}
\begin{equation}
    r_{k-1}=x_{k-1}=N_{k-3,k-1}=0.
\end{equation}
\end{claim}
\hspace{-6.15mm}\textit{Proof.} Let us prove that \(\mathcal{T}_M\) does not contain sets \(T\) such that 
\begin{align*}
    |T|-v(T)\in\left\{1,2\right\}, \ \ \ |T|<k.
\end{align*}
Arguing indirectly, assume that there is such \(T\in\mathcal{T}_M\). 
Since \(\mathcal{F}\) is shifted, we can add to \(T\) any \(k-1-|T|\) elements from \(G(M)\) and still have  \(T\in\mathcal{T}_M\). We add  elements such that \(|T|-v(T)\) remains fixed. Put \(c=v(T)<|T|=k-1\). Then \(c\geqslant k-3\).

Without loss of generality, \(T\) has non-empty intersection with each of \(B_1,\ldots,B_c\). Put \(a_i=|T\cap B_i|\) for \(i=1,\ldots,c\). Let us also put \(p_0=0\), \(p_i=a_1+\ldots+a_i\), \(i=1,\ldots,c\). Then \(p=p_c\leqslant k-1\), \(a_0:=k-p\geqslant 1\). Since \(c\geqslant k-3\), we have \(a_i\leqslant 3\), \(a_0=k-p\leqslant 3\).

We introduce convenient notation for  the elements of \(G(M)\):
\begin{align*}
    &B_i=\left\{b_i^1,\ldots,b_i^k\right\},\\
    &G_0=\left\{g_1,\ldots,g_{k-1}\right\},\\
    &T=\bigcup\limits_{i=1}^{c}\left\{b_i^{p_{i-1}+1},\ldots,b_i^{p_i}\right\}\cup\left\{g_{p+1},\ldots,g_{k-1}\right\},\\
    &\tilde B_i=B_i\setminus T,\ i=1,\ldots,c,\\
    &\tilde B_i=B_i,\ i=c+1,\ldots,k.
\end{align*}

Put \(\mu_i=\min\{j: p_j\geqslant i\}\), \(i=1,\ldots,p\). For a collection \(\overline{\pi}=\left(\pi_0,\pi_1,\ldots,\pi_k\right)\) of cyclic shifts on \(G_0\setminus T,\tilde B_1,\ldots,\tilde B_k\) (with arbitrary numeration) consider pairwise disjoint sets 
\begin{align*}
    &Q_{i,\overline\pi}=\left\{\pi_1(b_1^i),\ldots,\pi_{\mu_i-1}(b_{\mu_i-1}^i),\pi_0(g_i),\pi_{\mu_i+1}(b_{\mu_i+1}^i),\ldots,\pi_k(b_k^i)\right\},\ i=1,\ldots,p,\\
    &Q_{i,\overline\pi}=\left\{\pi_1(b_1^i),\ldots,\pi_k(b_k^i)\right\}, i=p+1,\ldots,k.
\end{align*}
\begin{wrapfigure}{l}{0.35\textwidth}
\captionsetup{font=scriptsize}
\includegraphics[width=0.98\linewidth]{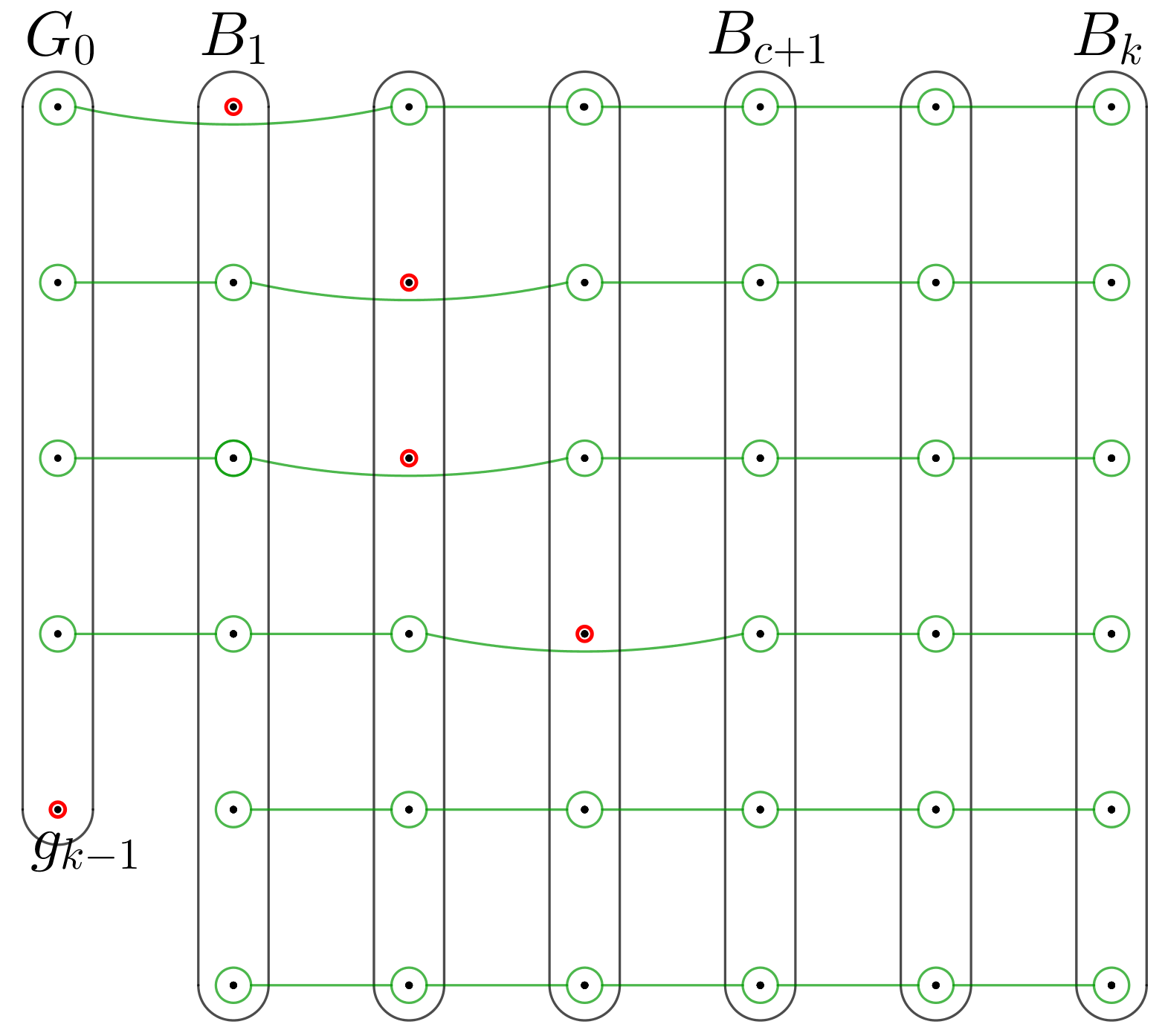} 
\caption{The elements of \(T\) are marked in red, \(p=k-2\). The transversal is depicted in green for identity permutations}
\label{fig:wrapfig}
\end{wrapfigure}
For \(i\leqslant p\) the set \(Q_{i,\overline\pi}\) is an almost full transversal, and for  \(i>p\) it is a full transversal. We cannot guarantee that for different \(\overline{\pi}\) the families \(\left\{Q_{i,\overline{\pi}}\right\}_{i=1}^k\) are disjoint. Still, we can bound the number of shifts \(\overline{\pi}\) such that \(Q=Q_{i,\overline{\pi}}\) for some \(i\in[k]\).


Take an arbitrary Q and assume that there is at least one  shift such that \( Q=Q_{i,\overline{\pi}}\). Assume that \(\overline{\sigma}\) satisfies \(Q_{i,\overline\pi}=Q_{j,\overline\sigma}\) for some \(j\). If \(i\leqslant p\) then we get \(j\leqslant p\) and \(\mu_i=\mu_j\), and if \(i>p\) then we get \(j>p\). Putting \(\mu_i=0\) for \(i>p\), in both cases we get \(\mu_i=\mu_j\). Also, in both cases all the shifts \(\sigma_l\), \(l=0,\ldots,k\), except \(\sigma_{\mu_i}\), are uniquely determined by \(j\). The number of \(j\) such that \(\mu_i=\mu_j\) is \(a_{\mu_i}\). Thus, the number of shifts \(\overline{\pi}\) such that \(Q=Q_{i,\overline{\pi}}\) is at most \(a_{\mu_i}(k-a_{\mu_i})\leqslant 3k\).

The total number of shifts \(\overline{\pi}\) is 
\begin{align*}
    &(k-a_0)\ldots(k-a_c)k^{k-c}\geqslant \left(1-\frac{1}{k}\right)^{\frac{k-1}{k-3}(a_0+\ldots+a_c)}k^{k+1}=\\
    &=\left(1-\frac{1}{k}\right)^{\frac{k^2-k}{k-3}}k^{k+1}>3k\cdot\left(1+\frac{3}{k}\right)2\varepsilon k^{k+1}.
\end{align*}
Therefore, for some collection of shifts \(\overline\pi\) all \(Q_i=Q_{i,\overline\pi}\) belong to \(\mathcal{T}_M\). Together with \(T\) and \(G_i\), \(i\notin M\), we get \(s+1\) pairwise disjoint sets in \(\mathcal{T}\), a contradiction.\qed

\subsection{Concluding the proof of Theorem~\ref{th2}}
Propositions~\ref{clRX} and~\ref{weightE} together with (\ref{ineqM}) imply
\begin{equation}
    W=W_3\leqslant\left(1+\frac{2}{k}\right)\frac{k^{k+6}\varepsilon}{6s^3}.
\end{equation}
Therefore, all but at most  \(\left(1+\frac{2}{k}\right)\frac{k^{k+6}\varepsilon}{6s^2}<k^{k-1}\) full and almost full transversals are missing.

To conclude the proof of the theorem, it is sufficient to prove the following. 
\begin{claim}
The family \(\mathcal{T}_M\) does not contain sets \(T\) such that \(|T|<k\).
\end{claim}
\begin{proof}
Assume that there is such a set \(T\in\mathcal{T}_M\). As in the proof of Proposition~\ref{clRX}, we assume that \(|T|=k-1\). Put \(v(T)=c\). W.l.o.g., assume  \(a_i=|T\cap B_i|>0\), \(i=1,\ldots,c\). Put \(a_0=|T\cap G_0|=k-a_1-\ldots-a_c\geqslant 1\). Repeating the argument from Proposition~\ref{clRX}, it is sufficient to check that 
\begin{align*}
    (k-a_0)\ldots(k-a_c)k^{k-c}\geqslant\max\limits_{i=0,\ldots,c}\left\{a_i(k-a_i)\right\}\cdot k^{k-1}.
\end{align*}
This is an easy calculation:
\begin{align*}
    &(k-a_0)\ldots(k-a_c)k^{k-c}=(k-a_i)k^k\prod\limits_{j\neq i}\left(1-\frac{a_j}{k}\right)\geqslant\\
    &\geqslant (k-a_i)k^k\left(1-\sum\limits_{j\neq i}\frac{a_j}{k}\right)=(k-a_i)k^{k-1}a_i.
\end{align*}
The last inequality holds for each \(i=0,\ldots,c\), which implies the proposition.
\end{proof}

We proved that all sets in  \(\mathcal{T}_M\) have cardinality \(k\).
This concludes the proof of inequality (\ref{ineq4}) and of Theorem~\ref{th2}.

\end{document}